\documentclass[12pt, reqno]{amsart}

\usepackage[utf8]{inputenc}
\usepackage{amssymb, amsmath, amsthm,amsfonts}

\usepackage{mathtools}
\usepackage{hyperref}
\usepackage{amsrefs}
\usepackage{amscd}  
\usepackage{fullpage}
\usepackage{stmaryrd}
\usepackage{mathrsfs}
\usepackage[all]{xy} 
\usepackage{MnSymbol}
\usepackage{multirow} 
\usepackage{array}
\usepackage{faktor}

\usepackage[usenames,dvipsnames]{xcolor}
\usepackage{tikz}
\usepackage{cleveref}

\usepackage{amsmath,amssymb,amsfonts,amsthm}
\usepackage{url,pdflscape,bigstrut,enumerate,enumitem,comment}
\usepackage{hyperref,booktabs,placeins}
\usepackage[dvipsnames]{xcolor}
\hypersetup{colorlinks=true,urlcolor=blue,citecolor=blue,linkcolor=blue}
\usepackage{tikz,pgfplots,pgfplotstable}
\usepackage{fullpage}

\newcommand{\Z}{\mathbb{Z}}
\newcommand{\Q}{\mathbb{Q}}

\usepgfplotslibrary{colorbrewer}
\pgfplotsset{cycle list/Dark2}

\pgfplotsset{compat=1.16}

\newtheorem{lemma}{Lemma}
\newtheorem{theorem}[lemma]{Theorem}

\newtheorem{cor}[lemma]{Corollary}
\newtheorem{conj}[lemma]{Conjecture}

\newtheorem{claim*}{Claim}

\theoremstyle{definition}
\newtheorem{remark}[lemma]{Remark}

\DeclareMathOperator{\im}{\mathrm{Im}}

\DeclareMathOperator{\re}{\mathrm{Re}}
\DeclareMathOperator{\Res}{\mathrm{Res}}

\newcommand{\e}{\epsilon}
\newcommand{\E}{\mathcal{E}}

\newcommand{\LL}{\Lambda}

\newcommand{\QQ}{\mathbb{Q}}

\newcommand{\ZZ}{\mathbb{Z}}
\newcommand{\FF}{\mathbb{F}}
\newcommand{\bO}{\mathcal{O}}

\newcommand{\ldl}{\frac{L'_{E}(s)}{L_{E}(s)}}

\newcommand{\rank}{\text{rank}}

\numberwithin{equation}{section}
\numberwithin{table}{section}

\title{From the Birch and Swinnerton-Dyer conjecture to Nagao's conjecture}

\author[1]{Seoyoung Kim}
\thanks{Research of the first author partially supported by a Coleman Postdoctoral Fellowship.}
\email{\textcolor{blue}{\href{mailto:sk206@queensu.ca}{sk206@queensu.ca}}}
\address{Department of Mathematics and Statistics, Queen's University, Kingston, ON, Canada  K7L 3N6}

\author[2]{M. Ram Murty \\ With an appendix by Andrew V. Sutherland}
\thanks{Research of the second author partially supported by NSERC Discovery grant.}
\email{\textcolor{blue}{\href{mailto:murty@queensu.ca}{murty@queensu.ca}}}
\address{Department of Mathematics and Statistics, Queen's University, Kingston, ON, Canada  K7L 3N6}

\begin{document}

\date{\today}
\maketitle

\begin{abstract}
Let $E$ be an elliptic curve over $\mathbb{Q}$ with discriminant $\Delta_E$. For primes $p$ of good reduction, let $N_p$ be the number of points modulo $p$ and write $N_p=p+1-a_p$. In 1965, Birch and Swinnerton-Dyer formulated a conjecture which implies
$$\lim_{x\to\infty}\frac{1}{\log x}\sum_{\substack{p\leq x\\ p\nmid \Delta_{E}}}\frac{a_p\log p}{p}=-r+\frac{1}{2},$$
where $r$ is the order of the zero of the $L$-function $L_{E}(s)$ of $E$ at $s=1$, which is predicted to be the Mordell-Weil rank of $E(\mathbb{Q})$. We show that if the above limit exits, then the limit equals $-r+1/2$. We also relate this to Nagao's conjecture.
\end{abstract}
\section{Introduction}
Let $E$ be an elliptic curve over $\QQ$ with discriminant $\Delta_E$ and conductor $N_E$. For each prime $p\nmid \Delta_E$, we write the number of points of $E\pmod{p}$ as
\begin{equation}
    N_{p}\coloneqq \#E(\FF_p)=p+1-a_p,
\end{equation}
where $a_p$ satisfies Hasse's inequality $|a_p|\leq 2\sqrt{p}$. For $p\mid \Delta_E$, we define $a_p=0, -1,$ or $1$ according as $E$ has additive reduction, split multiplicative reduction, or non-split multiplicative reduction at $p$ (For precise definitions of this terminology, we refer the reader to \cite{Sil1}*{p. 449}).\\

The $L$-function attached to $E$, denoted as $L_{E}(s)$ is then defined as an Euler product using this datum:
\begin{equation}
\label{L-function}
L_{E}(s)=\prod_{p\mid \Delta_E}\left(1-\frac{a_p}{p^s}\right)^{-1}\prod_{p\nmid \Delta_{E}}\left(1-\frac{a_{p}}{p^s}+\frac{p}{p^{2s}}\right)^{-1},
\end{equation}
which converges absolutely for $\re(s)>3/2$ by virtue of Hasse's inequality. Expanding the Euler product into a Dirichlet series, we write
\begin{equation}
    L_{E}(s)=\sum^{\infty}_{n=1}\frac{a_n}{n^s}.
\end{equation}
If we write $\alpha_p,\beta_p$ as the eigenvalues of the Frobenius morphism at $p$, for $p\nmid \Delta_E$, we can write $a_p=\alpha_p+\beta_p$, and our $L$-function can be re-written as 
\begin{equation}
    L_{E}(s)=\prod_{p\mid \Delta_E}\left(1-\frac{a_p}{p^s}\right)^{-1}\prod_{p\nmid \Delta_E}\left(1-\frac{\alpha_p}{p^s}\right)^{-1}\left(1-\frac{\beta_p}{p^s}\right)^{-1},
\end{equation}
and by Hasse's inequality $|\alpha_p|=|\beta_p|=\sqrt{p}$.\\

By the elliptic modularity theorem for semistable elliptic curves over $\mathbb{Q}$ by Wiles \cite{Wiles}, and its complete extension to all elliptic curves over $\mathbb{Q}$ by Breuil, Conrad, Diamond, and Taylor \cite{BCDT}, $L_{E}(s)$ extends to an entire function and satisfies a functional equation which relates $L_{E}(s)$ to $L_{E}(2-s)$. More precisely, if we define
\begin{equation}
    \LL_{E}(s)=N_E^{s/2}(2\pi)^{-s}\Gamma(s)L_{E}(s),
\end{equation}
then $\LL_{E}(s)$ is entire and satisfies the following functional equation:
\begin{equation}
\label{functionalequation}
    \LL_{E}(s)=w_{E}\LL_{E}(2-s),
\end{equation}
where $w_{E}\in\{1,-1\}$ is the root number of $E$. Since $\Gamma(s)$ has simple poles at $s=0,-1,\ldots$, $L_{E}(s)$ has trivial zeros at $s=0,-1,\ldots$. Thus, for $m=0,1,2,\cdots$, we have
\begin{equation}
    \underset{s=-m}{\Res}~\frac{L_{E}'(s)}{L_{E}(s)}=\underset{s=-m}{\Res}\left[\frac{\LL_{E}'(s)}{\LL_{E}(s)}-\frac{\Gamma'(s)}{\Gamma(s)}\right].
\end{equation}

This summarises our study of $E$ from the analytic perspective.\\

From the algebraic perspective, a celebrated theorem of Mordell states that the group of rational points $E(\QQ)$ is a finitely generated abelian group with rank $r_{M}$. In 1965, Birch and Swinnerton-Dyer \cite{BSD} conjectured that $L_{E}(s)$ has a zero of order $r$ at $s=1$. In other words, the algebraic rank $r_{M}$ equals the ``analytic rank" reflected as the order of zero at $s=1$ of $L_E(s)$. This conjecture is often referred to as the Birch and Swinnerton-Dyer conjecture.\\

However, before they formulated this conjecture in this form, Birch and Swinnerton-Dyer stated a stronger conjecture motivated by a heuristic ``local-global" principle: the rank should be reflected by ``modulo $p$" information for many primes $p$. More precisely, they conjectured that there is a constant $C_{E}$ such that
\begin{equation}
\label{OBSD}
    \underset{\substack{p<x \\ p\nmid \Delta_{E}}}\prod \frac{N_{p}}{p} \sim C_{E}(\log x)^{r},
\end{equation}
as $x\to\infty$. We refer to this as the original Birch and Swinnerton-Dyer conjecture (or OBSD for short).\\

Several authors have noted that the ``severity" of this conjecture in that it implies the analog of the Riemann hypothesis for $L_E(s)$, and much more. This was first announced by Goldfeld \cite{Goldfeld}. Kuo and Murty \cite{Kuo-Murty}*{Theorem 2, Theorem 3}, and K. Conrad \cite{Conrad}*{Theorem 1.3} independently noticed that \eqref{OBSD} goes well beyond the analog of the Riemann hypothesis for $L_E(s)$. They proved that \eqref{OBSD} is true if and only if
\begin{equation}
    \sum_{\substack{p^k\leq x\\p\nmid \Delta_E}}\frac{\alpha^k_p+\beta^k_p}{k}=o(x),
\end{equation}
as $x\to\infty$, or equivalently
\begin{equation}
\sum_{\substack{p^k\leq x \\ p\nmid\Delta_{E}}}(\alpha^{k}_{p}+\beta^{k}_{p})\log p=o(x\log x),
\end{equation}
as $x\to\infty$, whereas, the Riemann hypothesis for $L_E(s)$ is equivalent to the weaker assertion
\begin{equation}
\label{OBSD1}
\sum_{\substack{p^k\leq x\\p\nmid \Delta_{E}}}(\alpha^{k}_{p}+\beta^{k}_{p})\log p =\mathcal{O}(x(\log x)^2),    
\end{equation}
as $x\to\infty$.\\

If we return to the heuristic ``local-global" principle that perhaps motivated Birch and Swinnerton-Dyer to make their OBSD conjecture, we are led to formulate several gradations of their conjecture.\\

The first is that \eqref{OBSD} is equivalent to 
\begin{equation}
    \sum_{\substack{p^k\leq x \\ p\nmid \Delta_E}}\frac{\alpha^k_p+\beta^k_p}{kp^k}=-r\log\log x + A +o(1),
\end{equation}
for some constant $A$, as $x\to\infty$. If we weight each prime power $p^k$ by $\log p$ (following Chebysheff), we have that \eqref{OBSD} implies
\begin{equation}
    \sum_{\substack{p^k\leq x \\ p\nmid \Delta_E}}\frac{\alpha^k_p+\beta^k_p}{kp^k}\log p=-r\log x+\bO(1),
\end{equation}
as $x\to\infty$. In fact, this is equivalent to \eqref{OBSD}. However, the weaker assertion
\begin{equation}
    \sum_{\substack{p^k\leq x \\ p\nmid \Delta_E}}\frac{\alpha^k_p+\beta^k_p}{kp^k}\log p=-r\log x+o(\log x)
\end{equation}
already implies the analog of the Riemann hypothesis for $L_E(s)$ and still leads to an analytic determination of the rank $r$ using the ``local" data $a_p$.\\

There are good reasons to believe that \eqref{OBSD1} is not the optimal estimate. Montgomery \cite{Montgomery} and Gallagher \cite{Gallagher} have suggested in the context of the Riemann zeta function (but here applied to our context) that the error in \eqref{OBSD1} should be
\begin{equation}
    \bO(x(\log\log\log x)^2)
\end{equation}
or even the ``weaker'' $\bO(x(\log\log x)^2)$.\\

In either cases, it is conceivable that OBSD is true, even though it goes well beyond the analog of the Riemann hypothesis for $L_E(s)$.\\

The purpose of this note is to show that if the limit
\begin{equation}
    \lim_{x\to \infty}\frac{1}{\log x}\sum_{\substack{p^k\leq x \\ p\nmid \Delta_E}}\frac{\alpha^k_p+\beta^k_p}{kp^k}\log p
\end{equation}
exists, then the limit is $-r$. Moreover, this implies that if the limit
\begin{equation}
\lim_{x\to\infty}\frac{1}{\log x}\sum_{p<x}\frac{a_p\log p}{p}
\end{equation}
exists, then the Riemann hypothesis for $L_{E}(s)$ is true, and the limit is $-r+1/2$. We apply a technique of Cram\'er \cite{Cramer} to prove our theorem.

\section{Preliminaries}
For an elliptic curve $E$ defined over $\QQ$ with discriminant $\Delta_E$ and conductor $N_E$, we defined its $L$-function $L_{E}(s)$ in \eqref{L-function}. Hence, we can write its logarithmic derivative as
\begin{equation}
    -\frac{L'_{E}(s)}{L_{E}(s)}=\sum^{\infty}_{n=1}\frac{c_n\LL(n)}{n^s},
\end{equation}
where $\Lambda(n)$ is the von Mangoldt function, and 
\begin{equation}
\label{c_n}
c_{n} = \left\{ \begin{array}{ll}
\alpha^{m}_{p}+\beta^{m}_{p}, & \textrm{if $n=p^{m}$ and $p\nmid N$,}\\
a^{m}_{p}, & \textrm{if $n=p^m$ and $p\mid N$,}\\
0, & \textrm{otherwise.}
\end{array} \right.
\end{equation}
Let $\mathcal{C}$ be the (oriented) rectangle with vertices $c-iR,c+iR,-U+iR,-U-iR$, and its edges are denoted by $I_1, I_2, I_3$, and $I_4$ respectively, with $c=2$, an odd positive integer $R$ ($R$ is chosen so that it is not the ordinate of a zero of $L_E(s)$), and $U$ a positive non-integral number, we have by the Cauchy residue theorem,
\begin{align}
    \frac{1}{2\pi i}\int_{\mathcal{C}}-\frac{L'_{E}(s)}{L_{E}(s)}\frac{x^s}{s}ds &=-\underset{|\gamma|<R}{\sum}n_{\rho}\frac{x^{\rho}}{\rho}-\sum_{m<U}\underset{s=-m}\Res\left[\frac{L_{E}'(s)}{L_{E}(s)}\frac{x^s}{s}\right],
\end{align}
where the sum is over all zeros $\rho=\beta+i\gamma$ of $L_{E}(s)$, and $n_{\rho}$ is the multiplicity of each $\rho$. For any $n\geq 1$, Hence, by (truncated) Perron's formula, for $x$ not an integer, we get
\begin{align}
    \sum_{n\leq x}c_{n}\LL(n)&=\frac{1}{2\pi i}\int_{2-iR}^{2+iR} -\frac{L'_{E}(s)}{L_{E}(s)}\frac{x^s}{s}ds+\mathcal{O}\left(~\sum^{\infty}_{n=1}\left(\frac{x}{n}\right)^{2}|c_{n}\Lambda(n)|\cdot \min \left(1,\frac{1}{R|\log \frac{x}{n}|}\right)\right)\\
    &=-\underset{|\gamma|<R}{\sum}n_{\rho}\frac{x^{\rho}}{\rho}-\sum_{m<U}\underset{s=-m}\Res\left[\frac{L_{E}'(s)}{L_{E}(s)}\frac{x^s}{s}\right]+\frac{1}{2\pi i}\int_{\mathcal{C}\backslash I_1}\frac{L'_{E}(s)}{L_{E}(s)}\frac{x^s}{s}ds\\
    &\quad\quad\quad\quad\quad\quad\quad\quad\quad\quad\quad\quad+\mathcal{O}\left(~\sum^{\infty}_{n=1}\left(\frac{x}{n}\right)^{2}|c_{n}\Lambda(n)|\cdot \min \left(1,\frac{1}{R|\log \frac{x}{n}|}\right)\right), \label{error1}\\
    &=-\underset{|\gamma|<R}{\sum}n_{\rho}\frac{x^{\rho}}{\rho}+\sum_{m<U}\frac{x^{-m}}{m}+\frac{1}{2\pi i}\int_{\mathcal{C}\backslash I_1}\frac{L'_{E}(s)}{L_{E}(s)}\frac{x^s}{s}ds+E(x,U,R), \label{cauchy-integral}
\end{align}
where the first sum is over all nontrivial zeros $\rho=1+i\gamma$ of $L_{E}(s)$, and $n_{\rho}$ is the multiplicity of each $\rho$. $E(x,U,R)$ is the error term arising from the last term in \eqref{error1}. We first estimate the error term in \eqref{error1} and \eqref{cauchy-integral}:
$$E(x,U,R)=\mathcal{O}\left(~\sum^{\infty}_{n=1}\left(\frac{x}{n}\right)^{2}|c_{n}\Lambda(n)|\cdot \min \left(1,\frac{1}{R|\log \frac{x}{n}|}\right)\right).$$ 
We consider the following three parts of the sum in $E(x,U,R)$ separately:
\begin{align}
\label{error-estimation1}
    \left(\sum_{n<x/2}+\sum_{x/2<n<2x}+\sum_{n>2x}\right)\left(\frac{x}{n}\right)^{2}|c_{n}\Lambda(n)|\cdot \min \left(1,\frac{1}{R|\log \frac{x}{n}|}\right).
\end{align}
The first sum is for $n<x/2$, and hence we have $\log(x/n)>\log 2$, and
\begin{equation}
\label{error-estimation2}
    \sum_{n<x/2}\left(\frac{x}{n}\right)^{2}|c_{n}\Lambda(n)|\cdot \min \left(1,\frac{1}{R|\log \frac{x}{n}|}\right)=\bO\left(\frac{x^2}{R}\sum_{n<x/2}\frac{|c_n\Lambda(n)|}{n^2}\right)=\bO\left(\frac{x^2}{R}\right),
\end{equation}
for big enough $R$. Similarly, for the third sum of \eqref{error-estimation1}, the condition $n>2x$ implies $|\log(x/n)|>\log 2$, and hence
\begin{equation}
\label{error-estimation3}
    \sum_{n>2x}\left(\frac{x}{n}\right)^{2}|c_{n}\Lambda(n)|\cdot \min \left(1,\frac{1}{R|\log \frac{x}{n}|}\right)=\bO\left(\frac{x^2}{R}\right).
\end{equation}
Now, it remains to compute the second sum of \eqref{error-estimation1}
$$\sum_{x/2<n<2x}\left(\frac{x}{n}\right)^{2}|c_{n}\Lambda(n)|\cdot \min \left(1,\frac{1}{R|\log \frac{x}{n}|}\right).$$
We observe that $x/2<n<2x$ implies
$$-2<-\frac{x}{n}<\frac{x-n}{n}<\frac{x}{2n}<1,$$
and therefore
\begin{align}
\sum_{x/2<n<2x}\left(\frac{x}{n}\right)^{2}|c_{n}\Lambda(n)|\cdot \min \left(1,\frac{1}{R|\log \frac{x}{n}|}\right)=\bO\left(\sum_{x/2<n<2x}|c_n\Lambda(n)|\cdot\frac{n}{R|x-n|}\right).
\end{align}
We choose $x=N+1/2$, an integer plus half, then $x-n$ ranges
$$-N-\frac{1}{2}<x-n<0,$$
and thus, by denoting $j=|x-n|$, the sum can be rewritten by
\begin{equation}
    \sum_{x<n<2x}\frac{1}{|x-n|}=\sum_{j=1}^{2N+1}\frac{1}{j/2}=\leq\log(2N+1).
\end{equation}
One can similarly treat the range $x/2<n<x$ to deduce
\begin{equation}
\label{error-estimation4}    
    \sum_{x/2<n<2x}\left(\frac{x}{n}\right)^{2}|c_{n}\Lambda(n)|\cdot \min \left(1,\frac{1}{R|\log \frac{x}{n}|}\right)=\bO\left(\frac{2x\sqrt{x}(\log x)^2}{R}\right).
\end{equation}
In conclusion, we obtain from \eqref{error-estimation2}, \eqref{error-estimation3}, and \eqref{error-estimation4},
$$E(x,U,R)=\bO\left(\frac{x^2}{R}\right),$$
for suitable choice of $x$. The third term in \eqref{cauchy-integral}
$$\frac{1}{2\pi i}\int_{\mathcal{C}\backslash I_1}\frac{L'_{E}(s)}{L_{E}(s)}\frac{x^s}{s}ds$$
can be estimated with the usual method. We consider the following part of the sum from primes of good reduction:
\begin{equation}
    \psi_{E}(t)=\sum_{n\leq t}c_{n}\Lambda(n),
\end{equation}
and similar to the estimation of Cram\'er \cite{Cramer}, we have the following result:

\begin{theorem}
\label{main1}
Assuming the Riemann hypothesis for $L_{E}(s)$ is true, we obtain
\begin{equation}
    \lim_{x\to\infty}\frac{1}{\log x}\int^{x}_{2}\frac{\psi^{2}_{E}(t)}{t^3}dt=\sum_{\rho}\left|\frac{n_{\rho}}{\rho}\right|^2,
\end{equation}
where the sum is over all nontrivial zeros $\rho$ of $L_{E}(s)$, and $n_{\rho}$ is the multiplicity of each $\rho$.
\end{theorem}
\begin{proof}
From \eqref{error-estimation1}, and following estimations, we have
\begin{align*}
    \frac{\psi^2_{E}(t)}{t^{3}}& =\frac{1}{t^{3}}\left(\sum_{\rho}n_{\rho}\frac{t^{\rho}}{\rho}+\bO\left(\frac{t^2}{R}\right)\right)^2\\
    &=\sum_{\rho}n_{\rho}\sum_{\rho'}n_{\rho'}\frac{t^{\rho+\rho'-3}}{\rho\rho'}+\bO\left(\frac{t}{R^2}\right)+\bO\left(\frac{1}{R}\left|\sum_{\rho}n_{\rho}\frac{t^{\rho-1}}{\rho}\right|\right)\\
    &=\sum_{\rho}n_{\rho}\sum_{\rho'}n_{\rho'}\frac{t^{\rho+\rho'-3}}{\rho\rho'}+\bO\left(\frac{t}{R^2}\right)+\bO\left(\frac{t^{3/2}(\log t)^2}{R}\right),
\end{align*}
where the sums involving $\rho$ and $\rho'$ are taken over the zeros of $L_E(s)$ satisfying $|\im(\rho)|<R$ and $|\im(\rho')|<R$ (For simplifying the notation, we will drop these constraints throughout the proof), we assume the same condition for such sums throughout the proof. Thus, using the relation $\rho(2-\rho)=|\rho|^2$ and $\overline{\rho}=2-\rho$, and by choosing large enough $R$ in \eqref{cauchy-integral}, we obtain
\begin{align*}
    \int^{x}_{2}\frac{\psi^2_{E}(t)}{t^3}dt &=\sum_{\rho}\frac{n_{\rho}}{\rho}\sum_{\rho'}\frac{n_{\rho'}}{\rho'}\int^{x}_{2} t^{\rho+\rho'-3}dt+\bO(1)\\
    &=\sum_{\rho}\frac{n_{\rho}}{\rho}\left[ \sum_{\rho'=2-\rho} \frac{n_{\rho'}}{\rho'}\int^{x}_{2} t^{\rho+\rho'-3}dt+ \sum_{\rho'\neq 2-\rho}\frac{n_{\rho'}}{\rho'}\int^{x}_{2} t^{\rho+\rho'-3}dt\right]+\bO(1)\\
    &=\log x \sum_{\rho}\left|\frac{n_{\rho}}{\rho}\right|^2+\sum_{\rho}\frac{n_{\rho}}{\rho}\sum_{\rho'\neq 2-\rho}\frac{n_{\rho'}}{\rho'}\cdot\frac{x^{\rho+\rho'-2}-2^{\rho+\rho'-2}}{\rho+\rho'-2}+\bO(1).
\end{align*}
Note that $\rho'=2-\rho$ implies $\im(\rho')=-\im(\rho)$. We are now going to estimate the second term
$$\sum_{\rho}\frac{n_{\rho}}{\rho}\sum_{\rho'\neq 2-\rho}\frac{n_{\rho'}}{\rho'}\cdot\frac{x^{\rho+\rho'-2}-2^{\rho+\rho'-2}}{\rho+\rho'-2}.$$
Let $\eta>0$ be sufficiently small, which is independent from $x$, so that it is smaller than the smallest positive $\gamma$. We will estimate the sum separately for two different cases:
$$|\rho+\rho'-2|\geq \eta \quad \text{and} \quad |\rho+\rho'-2|<\eta.$$
First, by symmetry, it is sufficient to show that the following two sums converge for all zeros satisfying $|\rho+\rho'-2|\geq \eta$ and $\rho'\neq 2-\rho$:
\begin{align}
    \sum_{\gamma >0}\frac{n_{\rho}}{|\rho|}&\sum_{\gamma'>0}\frac{n_{\rho'}}{|\rho'|(\gamma+\gamma')}, \label{1}\\
    \sum_{\gamma >0}\frac{n_{\rho}}{|\rho|}& \sum_{o<\gamma'\leq \gamma-\eta}\frac{n_{\rho'}}{|\rho'|(\gamma-\gamma')} \label{2}.
\end{align}
Note that all sums in \eqref{1} and \eqref{2} are over zeros which satisfy $|\rho+\rho'-2|\geq \eta$. The convergence of the first sum is obvious, and thus, we consider the second sum.\\

Recall the following result by Selberg on the number of zeros of $L_{E}(s)$ in a bounded region \cite{Selberg}:
\begin{theorem}
\label{Selberg}
Let $N_{E}(T)$ be the number of zeros $\rho=\beta+i\gamma$ of $L_{E}(s)$ satisfying $0<\gamma\leq T$. Then
\begin{equation}
    N_{E}(T)=\frac{\alpha_{E}}{\pi}T(\log T +c)+S_{F}(T)+\bO(1),
\end{equation}
where $c$ is a constant, $\alpha_{E}$ is a constant which depends on $E$, and $S_{F}(T)=\bO(\log T)$.
\end{theorem}
Hence, we deduce $N_E(T+1)-N_E(T)\ll \log(T)$ from Theorem \ref{Selberg}. Combined with the sum \eqref{2}, we have
\begin{align*}
\sum_{o<\gamma'\leq \gamma-\eta}\frac{n_{\rho'}}{|\rho'|(\gamma-\gamma')}& =\left(\underset{0<\gamma'\leq \gamma^{\frac{2}{3}} }\sum + \underset{\gamma^{\frac{2}{3}}<\gamma'\leq \gamma-\gamma^{\frac{2}{3}} }\sum + \underset{\gamma-\gamma^{\frac{2}{3}}<\gamma'\leq \gamma-\eta}\sum\right)\frac{n_{\rho'}}{|\rho'|(\gamma-\gamma')}
\\
&=\bO\left(\frac{\log \gamma}{\gamma^{1/3}}\right),
\end{align*}
and we obtain
\begin{align*}
      \int^{x}_{2}\frac{\psi^2_{E}(t)}{t^3}dt&=\log x \sum_{\rho}\left|\frac{n_{\rho}}{\rho}\right|^2+\sum_{\rho}\underset{\substack{\rho' \neq 2-\rho\\ 0<|\gamma+\gamma'|<\eta}}{\sum}\frac{n_{\rho}n_{\rho'}(x^{\rho+\rho'-2}-2^{\rho+\rho'-2})}{\rho\rho'(\rho+\rho'-2)}+\bO(1)\\
      &=\log x \sum_{\rho}\left|\frac{n_{\rho}}{\rho}\right|^2+\sum_{\rho}\underset{\substack{\rho' \neq 2-\rho\\ 0<|\gamma+\gamma'|<\eta}}{\sum}\frac{n_{\rho}n_{\rho'}(x^{i(\gamma+\gamma')}-2^{i(\gamma+\gamma')})}{i\rho\rho'(\gamma+\gamma')}+\bO(1).
\end{align*}
From now on, we are going to show that, given arbitrarily small $\epsilon>0$, 
\begin{equation}
\label{ineq1}
    \left|~\sum_{\rho}\underset{\substack{\rho' \neq 2-\rho\\ 0<|\gamma+\gamma'|<\eta}}{\sum}\frac{n_{\rho}n_{\rho'}(x^{i(\gamma+\gamma')}-2^{i(\gamma+\gamma')})}{i\rho\rho'(\gamma+\gamma')}~\right|<\epsilon \log x,
\end{equation}
for $x>x_{0}$ with any choice of $x_{0}$. Considering the cases $|(\gamma+\gamma')\log x |\geq 1$ and $|(\gamma+\gamma')\log x |\leq 1$ separately, we observe that
\[\frac{x^{i(\gamma + \gamma')}-2^{i(\gamma + \gamma')}}{\gamma+\gamma'}\ll \min\left\{\frac{1}{|\gamma+\gamma'|},\log x\right\}.\]
We fix a large $T$ and let $\Delta \coloneqq \min \{|\gamma+\gamma'| : |\gamma|\leq T\}$. Then the left sum in \eqref{ineq1} is 
\begin{equation}
\label{sums-with-delta}
    \left|~\sum_{\rho}\underset{\substack{\rho' \neq 2-\rho\\ 0<|\gamma+\gamma'|<\eta}}{\sum}\frac{n_{\rho}n_{\rho'}(x^{i(\gamma+\gamma')}-2^{i(\gamma+\gamma')})}{i\rho\rho'(\gamma+\gamma')}~\right|\ll \frac{1}{\Delta}\sum_{~|\rho|,|\rho'|\leq T}\frac{1}{|\rho\rho'|}+ \sum_{\substack{|\rho|\geq T \\ |\rho'-\rho|<\eta}}\frac{\log x}{|\rho\rho'|},
\end{equation}
where the two sums on the right side of \eqref{sums-with-delta} are taken over $\rho, \rho'$ satisfying $0<|\gamma+\gamma'|<\eta$. 
Since the sum
\[ \sum_{\substack{|\rho|\geq T \\ |\rho'-\rho|<\eta}}\frac{\log x}{|\rho\rho'|}\]
converges as $T\to \infty$, we have by choosing large enough $x$
\[ \left|~\sum_{\rho}\underset{\substack{\rho' \neq 2-\rho\\ 0<|\gamma+\gamma'|<\eta}}{\sum}\frac{n_{\rho}n_{\rho'}(x^{i(\gamma+\gamma')}-2^{i(\gamma+\gamma')})}{i\rho\rho'(\gamma+\gamma')}~\right| \leq \e_{T}\log x + \bO_{T}(1).\]
Therefore, by letting $T\to \infty$ and assuming the Riemann hypothesis for $L_{E}(s)$, we obtain
\begin{equation}
    \lim_{x\to\infty}\frac{1}{\log x}\int^{x}_{2}\frac{\psi^{2}_{E}(t)}{t^3}dt=\sum_{\rho}\left|\frac{n_{\rho}}{\rho}\right|^2.
\end{equation}

\end{proof}

\begin{cor}
\label{main1-cor}
There exists $c>0$ such that for sufficiently large $x>0$, there exists $t\in[x,2x]$ which satisfies 
\begin{equation}
\label{cramer-bound}
    |\psi_{E}(t)|< ct\sqrt{\log t}.
\end{equation}
\end{cor}

\begin{proof}
Note that Theorem \ref{main1} implies
\begin{equation}
\label{main1-cor1}
    \int^{2x}_{x}\frac{\psi^{2}_{E}(u)}{u^3}~ du = o(\log x).
\end{equation}
Assume that for every $t\in[x,2x]$, we have $|\psi_{E}(t)|\geq ct\sqrt{\log t}$. This implies
\begin{equation}
    \int^{2x}_{x}\frac{\psi^2_{E}(u)}{u^3}du\geq\int^{2x}_{x}\frac{c^2u^2\log u}{u^3}du\geq c^2\log2\log x,
\end{equation}
which contradicts \eqref{main1-cor1}, for a sufficiently small $c$.
\end{proof}
\section{Birch and Swinnerton-Dyer conjecture and related works}
The Birch and Swinnerton-Dyer conjecture describes the rank $r_{M}$ of the Mordell-Weil group of $E$ and relates it to the order of vanishing of its $L$-function. The conjecture has been improved over time with numerical evidence, and there are several ways to describe their conjecture. In this paper, we are interested in the following version of the conjecture from \cite{BSD}, which we previously introduced as ``OBSD":
\begin{conj}[Birch and Swinnerton-Dyer]
\label{bsd}
For some constant $C_{E}$, we have
\begin{equation}
    \underset{\substack{p<x \\ p\nmid \Delta_{E}}}\prod \frac{N_{p}}{p} \sim C_{E}(\log x)^{r},
\end{equation}
where $r$ is the order of the zero of the $L$-function $L_{E}(s)$ of $E$ at $s=1$.
\end{conj}
Furthermore, Birch and Swinnerton-Dyer conjectured that the order of the zero of the $L$-function $L_{E}(s)$ of $E$ is equal to the rank $r_M$ of the Mordell-Weil group $E(\QQ)$ of $E$.\\

Kuo and Murty \cite{Kuo-Murty} and Conrad \cite{Conrad} independently showed that Conjecture \ref{bsd} is equivalent to an asymptotic condition of a sum involving the eigenvalues of the Frobenius at each prime. We cite the result from \cite{Kuo-Murty}*{Theorem 2 and 3} and \cite{Conrad}*{Theorem 1.3}:

\begin{theorem}
\label{kuo-murty-estimation}
Conjecture \ref{bsd} is true if and only if
\begin{equation}
\label{bsd-eq1}
    \sum_{\substack{p^k\leq x\\ p\nmid \Delta_{E}}}\frac{\alpha^{k}_{p}+\beta^{k}_{p}}{k}=o(x).
\end{equation}
Or equivalently, 
\begin{equation}
\label{bsd-eq2}
\sum_{\substack{p^k\leq x \\ p\nmid\Delta_{E}}}(\alpha^{k}_{p}+\beta^{k}_{p})\log p=o(x\log x).
\end{equation}
\end{theorem}

The Riemann hypothesis for $L_{E}(s)$ is equivalent to the following asymptotic condition of the sum \eqref{bsd-eq2}: 
\begin{equation}
\label{riemann-hypothesis}
\sum_{\substack{p^k\leq x\\p\nmid \Delta_{E}}}(\alpha^{k}_{p}+\beta^{k}_{p})\log p =\mathcal{O}(x(\log x)^2),    
\end{equation}
hence, as Kuo and Murty \cite{Kuo-Murty} and Conrad \cite{Conrad} pointed out, Conjecture \ref{bsd} is much deeper than the Riemann hypothesis for $L_{E}(s)$ according to our current knowledge.\\

Using Theorem \ref{main1}, we prove the following result:
\begin{theorem}
\label{main2}
Assume the Riemann hypothesis is true for $L_{E}(s)$. Then there is a sequence $x_n\in[2^{n},2^{n+1}]$ such that
\begin{equation}
    \lim_{n\to\infty}\frac{1}{\log x_n}\sum_{p<x_n}\frac{a_p\log p}{p}=-r+\frac{1}{2},
\end{equation}
where $r$ is the order of $L_{E}(s)$ at $s=1$.
\end{theorem}

\begin{proof}
For $x>1$, $d>3/2$ and any real number $a$, Perron's formula suggests
\begin{equation}
    \frac{1}{2\pi i}\int^{d+i\infty}_{d-i\infty}-\frac{L'_{E}(s)}{L_{E}(s)}\frac{x^s}{s-a} ~ds=x^a\sideset{}{'}\sum_{n\leq x}\frac{c_n\Lambda(n)}{n^a},
\end{equation}
where the dash on the sum means that the last term in the sum is weighted by $1/2$ if $x$ is an integer. We consider the case when $a=0,1$, and $d=2$, we get
\begin{equation}
    \frac{1}{2\pi i}\int^{c+i\infty}_{c-i\infty}-\frac{L'_{E}(s)}{L_{E}(s)}\frac{x^s}{s(s-1)}~ds=x\sum_{n\leq x}\frac{c_n\Lambda(n)}{n}-\sum_{n\leq x}c_n\Lambda(n).
\end{equation}
We write the expansion of $\frac{L'_{E}(s)}{L_{E}(s)}$ at $s=0$ and $s=1$ as
\begin{equation}
\ldl= \frac{r'}{s}+d'+\cdots, \quad \ldl=\frac{r}{s-1} +d+\cdots.   
\end{equation}
By following the residue computations as in \cite{Conrad}*{(6.8)}, we have
\begin{equation}
    \underset{s=\rho}\Res\left(-\frac{L'_{E}(s)}{L_{E}(s)}\frac{x^s}{s(s-1)}\right)=\left\{ \begin{array}{ll}
r'(\log x+2)+d' & \textrm{if $\rho=0$,}\\
-r x(\log x -2)-dx & \textrm{if $\rho=1$.}
\end{array} \right.
\end{equation}
Hence, considering the nontrivial zeros and the poles from $\Gamma(s)$, we get
\begin{equation}
\label{eq3}
    -rx\log x+\bO(x)=x\sum_{n\leq x}\frac{c_n\Lambda(n)}{n}-\sum_{n\leq x}c_n\Lambda(n),
\end{equation}
as $x$ tends to infinity, and we get
\begin{equation}
\label{eq4}
    \sum_{n\leq x}\frac{c_n\Lambda(n)}{n}=-r\log x+\frac{\sum_{n\leq x}c_n\Lambda(n)}{x}+\bO(1),
\end{equation}
where $c_n$ is defined as in \eqref{c_n}. On the other hand, we can separate the left hand side sum of \eqref{eq4}
\begin{align*}
    \sum_{n\leq x}\frac{c_n\Lambda(n)}{n}&=\sum_{p\leq x}\frac{a_p\log p}{p}+\sum_{p^2\leq x}\frac{(\alpha^2_p+\beta^2_p)\log p}{p^2}+o(1)\\
    &=\sum_{p\leq x}\frac{a_p\log p}{p}+\sum_{p\leq \sqrt{x}}\frac{(a^2_p-2p)\log p}{p^2}+o(1)\\
    &=\sum_{p\leq x}\frac{a_p\log p}{p}+\sum_{p\leq \sqrt{x}}\frac{a^2_p\log p}{p^2}-\sum_{p\leq \sqrt{x}}\frac{2\log p}{p}+o(1)\\
    &=\sum_{p\leq x}\frac{a_p\log p}{p}-\frac{1}{2}\log x+o(1).
\end{align*}
Note that the third equality follows from calculations in the proof of \cite{Kuo-Murty}*{Lemma 1} using the theory of the Rankin-Selberg convolution. Combining this with \eqref{eq4}, we obtain
\begin{equation}
    \sum_{p\leq x}\frac{a_p\log p}{p}=\left(-r+\frac{1}{2}\right)\log x+\frac{\sum_{n\leq x}c_n\Lambda(n)}{x}+\bO(1),
\end{equation}
where $r$ is the order of $L_{E}(s)$ at $s=1$.\\

Now, using Corollary \ref{main1-cor}, for $n\geq 2$, we can define a sequence by selecting $x_{n}\in [2^{n-1},2^{n}]$ such that
\begin{equation}
\label{cramer-bound-1}
    |\psi_{E}(x_n)|<cx_{n}\sqrt{\log x_n}.
\end{equation}
From \eqref{eq4} and \eqref{cramer-bound-1}, we have
\begin{equation}
    \sum_{p\leq x_n}\frac{a_p\log p}{p}=\left(-r+\frac{1}{2}\right)\log x_n +\bO(\sqrt{\log x_n}),
\end{equation}
and
\begin{equation}
    \frac{1}{\log x_n}\sum_{p\leq x_n}\frac{a_p\log p}{p}\rightarrow -r+\frac{1}{2} \quad \text{as $n\to \infty$},
\end{equation}
where $r$ is the order of $L_{E}(s)$ at $s=1$.
\end{proof}

Furthermore, Theorem \ref{main2} implies the following:

\begin{cor}
If the limit
\begin{equation}
\label{bsd-limit-cor}
\lim_{x\to\infty}\frac{1}{\log x}\sum_{p<x}\frac{a_p\log p}{p}
\end{equation}
exists, then the Riemann hypothesis for $L_{E}(s)$ is true, and the limit is $-r+1/2$.
\end{cor}

\begin{proof}
From the assumption, we can write
\begin{equation}
    \sum_{p<x}\frac{a_p\log p}{p}=K\log x +o(\log x),
\end{equation}
for some constant $K$. Then, \eqref{riemann-hypothesis} implies that the Riemann hypothesis for $L_{E}(s)$ is true, and thus Theorem \ref{main2} shows $K=-r+1/2$, on a subsequence $x_n\to\infty$. So, if the limit exists, it must be $-r+1/2$. 
\end{proof}

\begin{cor}
If the limit
\begin{equation}
\label{bsd-limit-cor-1}
\lim_{x\to\infty}\frac{1}{\log x}\sum_{p<x}\frac{a_p\log p}{p}
\end{equation}
exists, then Conjecture \ref{bsd} is true.
\end{cor}
\begin{proof}
We observe that
\begin{align}
\sum_{\substack{p^k\leq x \\ p\nmid\Delta_{E}}}(\alpha^{k}_{p}+\beta^{k}_{p})\log p&=\sum_{n\leq x}c_n\Lambda(n)\\
&=x\sum_{n\leq x}\frac{c_n\Lambda(n)}{n}+c'x\log x+\bO(x)\\
    &=x\sum_{p\leq x}\frac{a_p\log p}{p}-\frac{1}{2}x\log x+o(x)+c'x\log x+\bO(x)\\
    &=\left(-c'+\frac{1}{2}\right)x\log x + o(x\log x)-\frac{1}{2}x\log x +o(x)
    +c'x\log x+\bO(x)\\
    &=o(x\log x).
\end{align}
This is equivalent to Conjecture \ref{bsd}, which is remarked as in \eqref{bsd-eq2}.
\end{proof}

\section{Concluding remarks}

We make here a few remarks that relate our results to Nagao's conjecture. Recall that this conjecture focuses on the elliptic surface
\[\mathcal{E}: y^2=x^3+A(T)x+B(T)\]
with $A(T), B(T)\in\ZZ[T]$ and $\Delta(T)\coloneqq 4A(T)^3+27B(T)^2\neq 0$. For each $t\in\ZZ$ such that $\Delta(t)\neq 0$, we have an elliptic curve $\E_t$ defined over $\QQ$, and Nagao \cite{Nagao} defined the fibral average of the trace of Frobenius for each prime $p$ as follows:
\[A_p(\E)\coloneqq \frac{1}{p}\sum^{p}_{t=1}a_p(\E_t),\]
where $a_p(\E_t)$ is the trace of the Frobenius automorphism at $p$ (that we have been studying in the previous sections) of $\E_t$. In \cite{Nagao}, Nagao conjectured that 
\begin{equation}
    -\lim_{X\to\infty}\frac{1}{X}\sum_{p\leq X} A_{p}(\E)\log p=\rank~\E(\QQ(T)).
\end{equation}
Now the sum can be re-written as
\begin{equation}
\label{error-occurs}
    \sum_{p\leq X}\frac{1}{p}\sum_{t\leq p}a_{p}(\E_t)\log p=\sum_{t\leq X}\left(~\sum_{t\leq p\leq X}\frac{a_p(\E_t)\log p}{p}~\right),
\end{equation}
and from our analysis, for a fixed $t$, the inner sum is (ignoring error terms)
\[\left(-r_t+\frac{1}{2}\right)\log\frac{X}{t},\]
where $r_t=\rank ~ \E_t(\QQ)$, so one may expect
\[-\frac{1}{X}\sum_{p\leq X}A_p(\E)\log p \sim \frac{1}{X}\sum_{t\leq X}\left(r_t-\frac{1}{2}\right)\log\frac{X}{t}.\]
This suggests the following (modified) form of Nagao's conjecture:
\begin{equation}
\label{modified-nagao}
\lim_{X\to\infty}\frac{1}{X}\sum_{t\leq X}\left(r_t-\frac{1}{2}\right)\log\frac{X}{t}=  \textrm{rank}~\E(\QQ(T)).
\end{equation}

Now, note that
\[\sum_{t\leq X}\log \frac{X}{t}=X+\bO(\log X)\]
by a simple application of Stirling's formula. Thus, it would seem that
\[\lim_{X\to\infty}\frac{1}{X}\sum_{t\leq X}r_t\log\frac{X}{t}= \rank ~ \E(\QQ(T))+\frac{1}{2}.\]
By letting $R(u)=\sum_{t\leq u}r_t$, using Abel summation formula, it is not difficult to see that
\[\sum_{t\leq X}r_t\log\frac{X}{t}=\int^{X}_1\left(~\sum_{t\leq u}r_t\right)\frac{du}{u}.\]
Thus, perhaps we have
\[\int_{1}^{X}\frac{R(u)}{u}du\sim\left(\rank ~\E(\QQ(T))+\frac{1}{2}\right)X,\]
as $X\to\infty$. We now invoke the following elementary lemma: if $f(x)$ is a positive non-decreasing function such that, as $x\to\infty$
\[\int^{x}_{1}\frac{f(u)}{u}du \sim x,\]
then $f(x)\sim x$ as $x\to\infty$ \cite{Titchmarsh}*{3.7, Page 54}. Thus, our question becomes: is it true that
\begin{equation}
\label{average-rank}
\sum_{t\leq X}r_t \sim \left(\rank~\E(\QQ(T))+\frac{1}{2}\right)X \end{equation}
as $X\to\infty$, which can be viewed as a variant of Nagao's conjecture.

\begin{remark}
The upper bound of the sum \eqref{average-rank} has been studied by several authors assuming various standard conjectures. For instance, assuming OBSD(Conjecture \ref{bsd}), the Riemann hypothesis for $L$-series attached to elliptic curves, and Tate's conjecture for elliptic surfaces, Michel \cite{Michel} proved the upper bound
\begin{equation}
    \label{Michel-bound}
    \frac{1}{2X}\sum_{|t|\leq X}r_t\leq \left(\deg \Delta(T)+\deg N(T)-\frac{3}{2}\right)(1+o(1))
\end{equation}
as $X\to\infty$, where $N(T)$ denotes the conductor of $\E$. Moreover, assuming the same standard conjectures as Michel's result \cite{Michel}, Silverman \cite{Sil2} obtained the upper bound
\begin{equation}
    \label{Silverman-bound}
    \frac{1}{2X}\sum_{|t|\leq X}r_t\leq \left(\deg N(T)+\rank~\E(\QQ(T))+\frac{1}{2}\right)(1+o(1))
\end{equation}
as $X\to\infty$. Besides, Fermigier \cite{Fermigier} studied (93 different) one-parameter families of elliptic curves of generic rank $r=\rank~\E(\QQ(T))$ with $0\leq r \leq 4$. More precisely, let $t$ be an integer, then $32\%$ of the specialized curves $\E_t$ (defined over $\QQ$) had rank $r$, $48\%$ had rank $r+1$, $18\%$ had rank $r+2$, and only $2\%$ of the specialized curves had rank $r+3$. For every family of elliptic curves and bound which are considered in \cite{Fermigier}, Fermigier found the quantity
\begin{equation}
    \label{Fermigier-object}
    \frac{1}{2X}\sum_{|t|\leq X}\rank~
    \E_t(\QQ)-\rank~\E(\QQ(T))-\frac{1}{2}
\end{equation}
ranges from $0.08$ to $0.54$ and averages around $0.35$.
\end{remark}

\begin{remark}
\label{Silvermansremark}
The authors would like to thank Joseph Silverman for informing us of the following one parameter family of elliptic curves, which was first studied by Washington \cite{Washington}:
\begin{equation}
\label{Washingtonfamily}
    \E:y^2=x^3+Tx^2-(T+3)x+1,
\end{equation}
then $j(T)=256(T^2+3T+9)$, and $\E$, viewed as an elliptic curve defined over $\QQ(T)$ has $\rank~\E(\QQ(T))=1$. Interestingly, Rizzo \cite{Rizzo}*{Theorem 1} proved that the family has large bias in its fibral root numbers. More precisely, the root numbers of the family, which is previously defined in \eqref{functionalequation}, of each fiber is
\[w_{\E_t}=-1, \quad\text{for every $t\in\mathbb{Z}$.}\]
Hence, via the Parity conjecture (or the Birch and Swinnerton-Dyer conjecture), we can expect the rank of all fibers will be odd, and in particular, positive with infinitely many rational points. On the other hand, via Silverman's specialization theorem, it is known that $\rank~ \E(\QQ(T))\leq \rank ~\E_{t}(\QQ)$ for all but finitely many $t$. Furthermore, one conjectures that $\rank~\E_t(\QQ)$ is equal to either $\rank~ \E(\QQ(T))$ or $\rank~ \E(\QQ(T))+1$ up to a zero density subset of $\QQ$, depending on the parity given by the root numbers. This tells us that, outside of a zero density subset of $\QQ$, the fibers of $\E$ have odd rank, and most likely have rank $1$. Therefore, it is likely that
\begin{equation}
    \lim_{X\to\infty}\frac{1}{2X}\sum_{|t|\leq X}\rank~ \E_{t}(\QQ)=\rank~\E(\QQ(T)),
\end{equation}
where the extra $1/2$ disappears since the fibral parity of the family $\E$ has large bias.
\end{remark}

Based on the heuristics suggested as in \eqref{average-rank}, Remark \ref{Silvermansremark}, and in \cites{Fermigier, Sil2}, the above remark, we propose the following conjecture:
\begin{conj}
\label{conjecture1}
We define
\begin{equation}
    \mathcal{T}\coloneqq \left\{t\in \ZZ: w_{\E_t}=(-1)^{\rank~ \E(\QQ(T))+1} \right\}.
\end{equation}
Then we have
\begin{equation}
    \lim_{X\to\infty}\frac{1}{2X}\sum_{|t|\leq X}r_{t}=\rank~ \E(\QQ(T))+\delta(\mathcal{T})
\end{equation}
where $\delta(\mathcal{T})$ is the natural density of $\mathcal{T}$ as a subset of $\ZZ$.
\end{conj}
Note that the conjecture reflects the heuristic \eqref{modified-nagao}, Remark \ref{Silvermansremark}, as well as the experimental results of Fermigier: more precisely, the results in \cite{Fermigier}*{Tableau 2} shows that the quantity \eqref{Fermigier-object} does not tend to depend significantly as the other invariants of the elliptic curves change, such as $\rank ~\E(\QQ(T)),~\deg N(T),$ and $\deg \Delta (T)$, which appear in the known upper bounds, as in \eqref{Michel-bound} and \eqref{Silverman-bound}. The work of Fermigier also suggests that the average of the error term in \eqref{error-occurs} is bounded.\\

A curious consequence of this conjecture is that the specializations $\E_t$ with $r_t$ large are very rare. It may be possible to prove such consequences of the conjecture by other methods.\\

Also, note that the bound \eqref{Michel-bound} and \eqref{Silverman-bound} have been also considered for a family of abelian varieties over $\QQ$ by Wazir \cite{Wazir}. More concretely, let $\pi:\mathcal{A}\rightarrow \mathbb{P}^1$ be a propler flat morphism of smooth projective varieties defined over $\QQ$, with an abelian variety $A$ over $\QQ(T)$ as its generic fiber. Then, by assuming standard conjectures as in \cite{Sil2}, Wazir obtains the following bound
\begin{equation}
    \frac{1}{2X}\sum_{|t|\leq X}\rank \mathcal{A}_{t}(\QQ)\leq \left(\frac{\mathcal{L}_X}{2X\log X}+\rank \mathcal{A}(\QQ(T))+\frac{g}{2}\right)(1+o(1))
\end{equation}
as $X\to\infty$, where $\mathcal{A}_{t}$ is the fiber at $t\in\mathbb{P}^1$, which is an abelian variety defined over $\QQ$, and
\begin{equation}
    \mathcal{L}_{X}=\sum_{|t|\leq X}\log N_{\mathcal{A}_{t}},
\end{equation}
where $N_{\mathcal{A}_{t}}$ is the conductor of $\mathcal{A}_t$. For more detailed definition of $\mathcal{L}_{X}$ and the conductor, please refer to \cite{Wazir}*{1.1}. Hence, one can expect similar conjectural estimation as Conjecture \ref{conjecture1} involving $g/2$ in place of $1/2$. We leave the details for future studies.  
\section*{acknowledgements}
The authors would like to thank Marc Hindry and Joseph Silverman for their helpful suggestions. The authors thank Noam Elkies for providing several of the elliptic curves used in the appendix. The authors also wish to thank the anonymous referee(s) for helpful suggestions on the earlier version of this paper.   


\newpage

\section*{Appendix}
\begin{center}
{by Andrew V. Sutherland}\footnote{Department of Mathematics, Massachusetts Institute of Technology, 77 Massachusetts Ave., Cambridge, MA 02139, USA; email: \texttt{drew@math.mit.edu}, URL: \url{https://math.mit.edu/~drew}.  Supported by Simons Foundation grant 550033.}
\end{center}

Let $a_p(E)$ denote the Frobenius trace of an elliptic curve $E/\Q$ at a prime $p$.  Figures~\ref{fig:one}, \ref{fig:two}, \ref{fig:three} plot the sums
\[
S(x):=\frac{1}{\log x}\sum_{p\le x,\ p\nmid \Delta_E} \!\!\!\!\frac{a_p(E)\log p}{p}
\]
for elliptic curves $E$ of discriminant $\Delta_E$ and various ranks; See Table~\ref{tab:sources} for a list of the curves and their sources.  These sums are conjectured to converge to $r-1/2$ as $x\to \infty$, where $r$ is the analytic rank of $L_E(s)$.  The ranks $r_E$ listed in Table~\ref{tab:sources} are lower bounds on the Mordell-Weil rank that are also upper bounds on the analytic rank under the Generalized Riemann Hypothesis (GRH), and equal to both the Mordell-Weil rank and the analytic rank under the Birch and Swinnerton--Dyer conjecture (BSD).  Ranks listed with no asterisk are equal to the Mordell-Weil rank; for those marked with a single (resp.\ double) asterisk this equality is conditional on GRH (resp.\ GRH and BSD).  Lower bounds on the Mordell-Weil rank were confirmed by verifying the existence of $r_E$ independent points using the N\'eron-Tate height pairing, while GRH-based upper bounds on the analytic rank were confirmed using Bober's method \cite{Bober13}.  GRH-based upper bounds on the Mordell-Weil rank were confirmed using \texttt{magma} \cite{magma} for $r_E \le 19$; for $r_E\ge 20$ we rely on the results of \cite{KSW}.  Exact values of Mordell-Weil ranks were confirmed using Cremona's \texttt{mwrank} \cite{mwrank} package for $r_E\le 11$; for ranks $r_E\ge 12$ with no asterisk we rely on computations reported in the listed sources. The curves of rank $r_E\le 11$ have conductors $N_E$ that are close to the smallest possible \cite{EW04}.  This is not likely to be true for the curves for rank $r_E\ge 12$, but we chose curves of smaller conductor when several were available.  In many cases the curves we list are not the first known curve of that rank; see \cite{Dujella2} for a history of rank records.

The sums $S(x)$ plotted for $x\le B = 10^{12}$ in Figure~\ref{fig:one} were computed using the \texttt{smalljac} software library~\cite{KS08} with some further optimizations described in \cites{S07,S11}.  The expected time complexity of this approach is $O(B^{5/4}\log B\log\log B)$.  This is asymptotically worse than both the $O(B(\log B)^{4+o(1)})$ expected time complexity (under GRH) of using the Schoof-Elkies-Atkin algorithm \cite{SS15} and the $O(B(\log B)^3)$ time complexity of an average polynomial-time approach \cite{HS16}, but it is practically much faster for $B=10^{12}$; it took approximately 100 core-days per curve to compute the $S(x)$ plots shown in Figure~\ref{fig:one}.

Figures~\ref{fig:two} and~\ref{fig:three} show similar plots for Mordell curves \cite{Mordell14} $y^2=x^3+k$ of ranks $r_E=0, 1, 2, \ldots, 17$ and congruent number curves \cite{Mordell22} $y^2=x^3-n^2x$ of ranks $r_E=0, 1, 2, \ldots, 7$. The corresponding curves listed in Table~\ref{tab:sources} were chosen to minimize $|k|$ and $n$ among those of a given rank which have appeared in the literature (we chose the sign of $k$ yielding a smaller conductor).  These are not necessarily the first curves of these forms found to achieve these ranks; see \cites{LQ88,Quer87,Womack00} for some earlier examples.

The Mordell curves and congruent number curves have $j$-invariants $0, 1728$ (respectively), and thus admit (potential) complex multiplication by the ring of integers $\mathcal O$ of  $K=\Q(\zeta_3), \Q(i)$.
To efficiently compute $a_p(E)=\mathrm{tr} (\psi_E(\mathfrak p))$ we compute the trace of the Hecke character $\psi_E$ corresponding to $E$ evaluated at a prime~$\mathfrak p$ of $K$ above~$p$; this trace is necessarily zero when $p$ is inert.
For each prime $p\le B$ of good reduction for $E$ that splits in $\mathcal O$ we use Cornacchia's algorithm to compute all integer solutions $(t,v)$ to the norm equation $4p=t^2-v^2\mathrm{disc}(\mathcal O)$.  We then apply the algorithm of Rubin and Silverberg \cite{RS10} to determine the correct choice of $t=a_p$. The algorithm in \cite{RS10} determines the correct twist of an ordinary elliptic curve over~$\mathbb F_p$ with a given $j$-invariant, endomorphism ring and Frobenius trace, but it can also be used to determine the Frobenius trace of an ordinary elliptic curve over $\mathbb F_p$ whose endomorphism ring is known.  This yields an algorithm to compute $a_p(E)$ in $O((\log p)^2\log \log p)$ expected time, meaning we can compute $S(x)$ for $x\le B$ in $O(B \log B\log\log B)$ expected time.  This makes it feasible to plot $S(x)$ for $x\le B=10^{15}$ in Figures~\ref{fig:two} and~\ref{fig:three} in roughly the same time required for $B=10^{12}$ in Figure~\ref{fig:one}, about 100 core-days per curve.

We end with a note of caution regarding the interpretation of these plots as evidence supporting the conjectured convergence of $S(x)$.  The methods used to find the higher rank curves shown in these plots typically use $S(x)$ or a closely related sum as a heuristic method to identify elliptic curves of potentially high rank; see~\cites{Campbell99,Mestre86,Nagao94}.  Most of the curves of rank $r_E\ge 12$ listed in Table~\ref{tab:sources} were discovered precisely because a partial sum related to $S(x)$ suggested they should have large ranks. This is less of a concern for the lower rank curves where searches have been more exhaustive in an effort to minimize $N_E$, $|k|$, or $n$.

\begin{table}[tbh!]
\footnotesize
\renewcommand{\arraystretch}{0.75}
\begin{tabular}[width=\textwidth]{lccll}
$r_E$ & $E(\Q)_{\rm tors}$ & $\log N_E$ & $[a_1,a_2,a_3,a_4,a_6]$ & source \\\midrule
0 & trivial & 2.398 & \href{https://www.lmfdb.org/EllipticCurve/Q/11a3/}{$[0,-1,1,0,0]$} & Birch and Kuyk 1972 \cite{BK75}\\
1 & trivial & 3.611 & \href{https://www.lmfdb.org/EllipticCurve/Q/37a1/}{$[0,0,1,-1,0]$} & Birch and Swinnerton-Dyer 1963 \cite{BSD63}\\
2 & trivial & 5.964 & \href{https://www.lmfdb.org/EllipticCurve/Q/389a1/}{$[0,1,1,-2,0]$} & Brumer and Kramer 1977 \cite{BK77}\\
3 & trivial & 8.532 & \href{https://www.lmfdb.org/EllipticCurve/Q/5077a1/}{$[0,0,1,-7,\ 6]$} & Brumer and Kramer 1977 \cite{BK77}\\
4 & trivial & 12.365 & \href{https://www.lmfdb.org/EllipticCurve/Q/234446a1/}{$[1,-1,0,-79,\ 289]$} & Connell 1999 \cites{Connell99b,EW04}\\
5 & trivial & 16.762 & \href{https://www.lmfdb.org/EllipticCurve/Q/19047851/a/1}{$[0,0,1,-79,\ 342]$} & Brumer and McGuinness 1990 \cite{BM90}\\
6 & trivial & 22.370 & \href{https://math.mit.edu/~drew/rk6.html}{$[1,1,0,-2582,\ 48720]$}& Elkies and Watkins 2004 \cite{EW04}\\
7 & trivial & 26.670 & \href{https://math.mit.edu/~drew/rk7.html}{$[0,0,0,-10012,\ 346900]$}& Elkies and Watkins 2004 \cite{EW04}\\
8 & trivial & 33.151 & \href{https://math.mit.edu/~drew/rk8.html}{$[1,-1,0,-106384,\ 13075804]$}& Elkies and Watkins 2004 \cite{EW04}\\
9 & trivial & 38.008 & \href{https://math.mit.edu/~drew/rk9.html}{$[1,-1,0,-135004,\ 97151644]$}& Elkies and Watkins 2004 \cite{EW04}\\
10 & trivial & 43.768 & \href{https://math.mit.edu/~drew/rk10.html}{$[0,0,1,-16312387,\ 25970162646]$}& Elkies and Watkins 2004 \cite{EW04}\\
11 & trivial & 51.246 & \href{https://math.mit.edu/~drew/rk11.html}{$[0,0,1,-16359067,\ 26274178986]$}& Elkies and Watkins 2004 \cite{EW04}\\
$12^*$ & trivial & 67.767 & \href{https://web.math.pmf.unizg.hr/~duje/tors/rk12.html}{$[0,0,1,-634\cdots 647,\ 193\cdots 036]$} & Mestre 1982 \cite{Mestre82}\\
$13^*$ & trivial & 99.778 & \href{https://math.mit.edu/~drew/rk13.html}{$[1,0,0,-560\cdots 540,\ 529\cdots 600]$} & Nagao 1994 \cite{Nagao94}\\
$14^*$ & trivial & 86.484 & \href{https://web.math.pmf.unizg.hr/~duje/tors/rk14.html}{$[0,0,1,-224\cdots 757,\ 132\cdots 406]$} & Mestre 1986 \cite{Mestre86}\\
$15^*$ & trivial & 129.440 & \href{https://web.math.pmf.unizg.hr/~duje/tors/rk15.html}{$[1,0,0,-209\cdots 485,\ 266\cdots 897]$} & Mestre 1992 \cite{Mestre92}\\
16 & $\Z/2\Z$ & 139.095 & \href{https://web.math.pmf.unizg.hr/~duje/tors/z2old1415161718.html}{$[1,0,0,888\cdots 054,\ 398\cdots 0420]$} & Dujella 2009 \cite{Dujella1}\\
$17^*$ & trivial & 136.210 & \href{https://web.math.pmf.unizg.hr/~duje/tors/rk17.html}{$[0,1,0,-184\cdots 145,\ 966\cdots 743]$} & Nagao 1992 \cite{Nagao92}\\
18 & $\Z/2\Z$ & 149.798 & \href{https://web.math.pmf.unizg.hr/~duje/tors/z2old1415161718.html}{$[1,0,0,-171\cdots 215,\ 445\cdots 817]$} & Elkies 2009 \cite{Dujella1}\\
$19^*$ & trivial & 149.986 & \href{https://web.math.pmf.unizg.hr/~duje/tors/rk19.html}{$[1,-1,1,-206\cdots 978,\ 328\cdots 881]$} & Fermigier 1992 \cite{F92}\\
$20^*$ & trivial & 170.088 & \href{https://web.math.pmf.unizg.hr/~duje/tors/rk20.html}{$[1,0,0,-431\cdots 166,\ 515\cdots 196]$} & Nagao 1993 \cite{Nagao93}\\
$21^*$ & trivial & 196.680 & \href{https://web.math.pmf.unizg.hr/~duje/tors/rk21.html}{$[1,1,1,-215\cdots 835\ -194\cdots 535]$} & Nagao and Kouya 1994 \cite{NK94}\\
$22^*$ & trivial & 182.725 & \href{https://web.math.pmf.unizg.hr/~duje/tors/rk22.html}{$[1,0,1,-940\cdots 864,\ 107\cdots 362]$} & Fermigier 1996 \cite{F96}\\
$23^*$ & trivial & 205.061 & \href{https://web.math.pmf.unizg.hr/~duje/tors/rk23.html}{$[1,0,1,-192\cdots 723,\ 326\cdots 006]$} & Martin and McMillen 1998 \cite{MM98}\\
$24^*$ & trivial & 219.927 & \href{https://web.math.pmf.unizg.hr/~duje/tors/rk24.html}{$[1,0,1,-120\cdots 374,\ 504\cdots 116]$} & Martin and McMillen 2000 \cite{MM00}\\
$25^{**}$ & trivial & 229.186 & \href{https://math.mit.edu/~drew/rk25.html}{$[1,0,0,-122\cdots 200,\ 523\cdots 000]$} & Elkies 2006 \cite{Elkies21a}\\
$26^{**}$ & trivial & 247.860 & \href{https://math.mit.edu/~drew/rk26.html}{$[1,0,0,-271\cdots 190,\ 167\cdots 092]$}& Elkies 2006 \cite{Elkies21a}\\
$27^*$ & trivial & 287.013 & \href{https://web.math.pmf.unizg.hr/~duje/tors/rk27.html}{$[1,0,0,-556\cdots 970,\ 161\cdots 956]$} & Elkies 2006 \cite{KSW}\\
$28^*$ & trivial & 368.407 & \href{https://web.math.pmf.unizg.hr/~duje/tors/rk28.html}{$[1,-1,1,-200\cdots 502,\ 344\cdots 429]$} & Elkies 2006 \cite{Elkies06}\\
\midrule
0 & $\Z/6\Z$ & 3.584 & \href{https://www.lmfdb.org/EllipticCurve/Q/36a1}{$[0,0,0,0,1]$} & Euler 1747 \cite{Euler}*{Theorem 10}\\
1 & trivial & 7.455 & \href{https://www.lmfdb.org/EllipticCurve/Q/1728a1}{$[0,0,0,0,2]$} & Billing 1938 \cite{Billing38}, Cassels 1950 \cite{Cassels50}\\
2 & trivial & 9.478 & \href{https://www.lmfdb.org/EllipticCurve/Q/13068k1}{$[0,0,0,0,-11]$} & Billing 1938 \cite{Billing38}, Cassels 1950 \cite{Cassels50}\\
3 & trivial & 14.137 & \href{https://math.mit.edu/~drew/mrk3.html}{$[0,0,0,0,113]$} & Birch and Swinnerton-Dyer 1963 \cite{BSD63}\\
4 & trivial & 18.872 & \href{https://math.mit.edu/~drew/mrk4.html}{$[0,0,0,0,2089]$} &  Gebel, Petho, Zimmer 1998 \cite{GPZ98}\\
5 & trivial & 24.083 & \href{https://math.mit.edu/~drew/mrk5.html}{$[0,0,0,0,-28279]$} &  Gebel, Petho, Zimmer 1998 \cite{GPZ98}\\
6 & trivial & 31.540 & \href{https://math.mit.edu/~drew/mrk6.html}{$[0,0,0,0,1358556]$} & Womack 2000 \cite{Womack00}\\
7 & trivial & 39.296 & \href{https://math.mit.edu/~drew/mrk7.html}{$[0,0,0,0,-56877643]$} &  Womack 2000 \cite{Womack00}\\
8 & trivial & 45.493 & \href{https://math.mit.edu/~drew/mrk8.html}{$[0,0,0,0,-2520963512]$} & Womack 2000 \cite{Womack00} \\
9 & trivial & 52.637 & \href{https://math.mit.edu/~drew/mrk9.html}{$[0,0,0,0,-44865147851]$} & Elkies 2009 \cite{Elkies09} \\
10 & trivial & 61.126 & \href{https://math.mit.edu/~drew/mrk10.html}{$[0,0,0,0,3612077876156]$} & Elkies 2009 \cite{Elkies09}\\ 
11 & trivial & 72.659 & \href{https://math.mit.edu/~drew/mrk11.html}{$[0,0,0,0,-998820191314747]$} & Elkies 2009 \cite{Elkies09}\\
12 & trivial & 80.089 & \href{https://math.mit.edu/~drew/mrk12.html}{$[0,0,0,0,41025014649039529]$} & Elkies 2009 \cite{Elkies09}\\
13 & trivial & 87.294 & \href{https://math.mit.edu/~drew/mrk13.html}{$[0,0,0,0,48163745551486811536]$} & Elkies 2009 \cite{Elkies09}\\
14 & trivial & \!\!\!103.188\!\!\! & \href{https://math.mit.edu/~drew/mrk14.html}{$[0,0,0,0,785\cdots 336]$} & Elkies 2009 \cite{Elkies09}\\
15 & trivial & \!\!\!122.905\!\!\! & \href{https://math.mit.edu/~drew/mrk15.html}{$[0,0,0,0,469\cdots 417]$} & Elkies 2009 \cite{Elkies09}\\
16 & trivial & \!\!\!136.203\!\!\! & \href{https://math.mit.edu/~drew/mrk16.html}{$[0,0,0,0,116\cdots 888]$} & Elkies 2016 \cite{Elkies16a}\\
$17^*$ & trivial & \!\!\!155.363\!\!\! & \href{https://math.mit.edu/~drew/mrk17.html}{$[0,0,0,0,-908\cdots 363]\!\!\!\!$} & Elkies 2016 \cite{Elkies16b}\\
\midrule
0 & $(\Z/2\Z)^2$ & 3.466 & \href{https://www.lmfdb.org/EllipticCurve/Q/32a2}{$[0,0,0,-1^2,0]$} & Fermat 1659 \cite{Fermat1659}\\
1 & $(\Z/2\Z)^2$ & 6.685 & \href{https://www.lmfdb.org/EllipticCurve/Q/800a1}{$[0,0,0,-5^2,0]$} & Billing 1938 \cite{Billing38}\\
2 & $(\Z/2\Z)^2$ & 9.825 & \href{https://www.lmfdb.org/EllipticCurve/Q/18496i2}{$[0,0,0,-34^2,0]$} & Wiman 1945 \cite{Wiman45}\\
3 & $(\Z/2\Z)^2$ & 17.041 & \href{https://math.mit.edu/~drew/crk3.html}{$[0,0,0,-1254^2,0]$} & Wiman 1945 \cite{Wiman45}\\
4 & $(\Z/2\Z)^2$ & 23.341 & \href{https://math.mit.edu/~drew/crk4.html}{$[0,0,0,-29274^2,0]$} & Wiman 1945 \cite{Wiman45}\\
5 & $(\Z/2\Z)^2$ & 38.850 & \href{https://math.mit.edu/~drew/crk5.html}{$[0,0,0,-48272239^2,0]$} & Rogers 2004 \cite{Rogers04}\\
6 & $(\Z/2\Z)^2$ & 47.997 & \href{https://math.mit.edu/~drew/crk6.html}{$[0,0,0,-6611719866^2,0]$} & Rogers 2004 \cite{Rogers04}\\
7 & $(\Z/2\Z)^2$ & 58.275 & \href{https://math.mit.edu/~drew/crk7.html}{$[0,0,0,-797507543735^2,0]$} & Rogers 2004 \cite{Rogers04}\\
\bottomrule
\end{tabular}
\smallskip

\caption{Elliptic curves listed by Mordell-Weil rank $r_E$.  Asterisks (double asterisks) indicate ranks conditional on GRH (GRH and BSD).}\label{tab:sources}
\end{table}

\FloatBarrier

\begin{figure}[tbh!]
\includegraphics[scale=0.95]{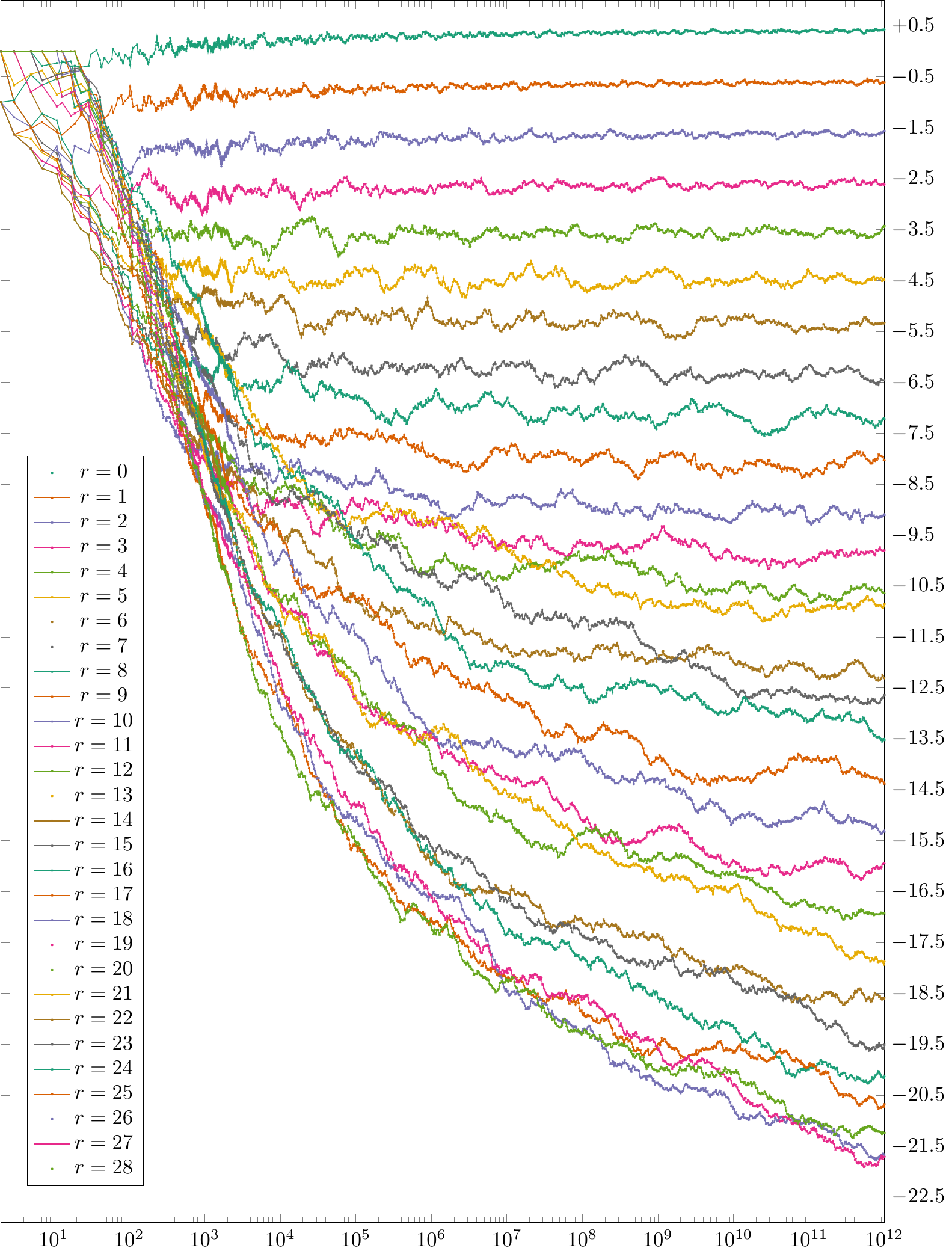}
\vspace{-8pt}
\caption{$S(x)$ plot for elliptic curves of rank $r=r_E$ listed in Table~\ref{tab:sources}.}\label{fig:one}
\end{figure}

\begin{figure}[t]
\includegraphics[scale=0.95]{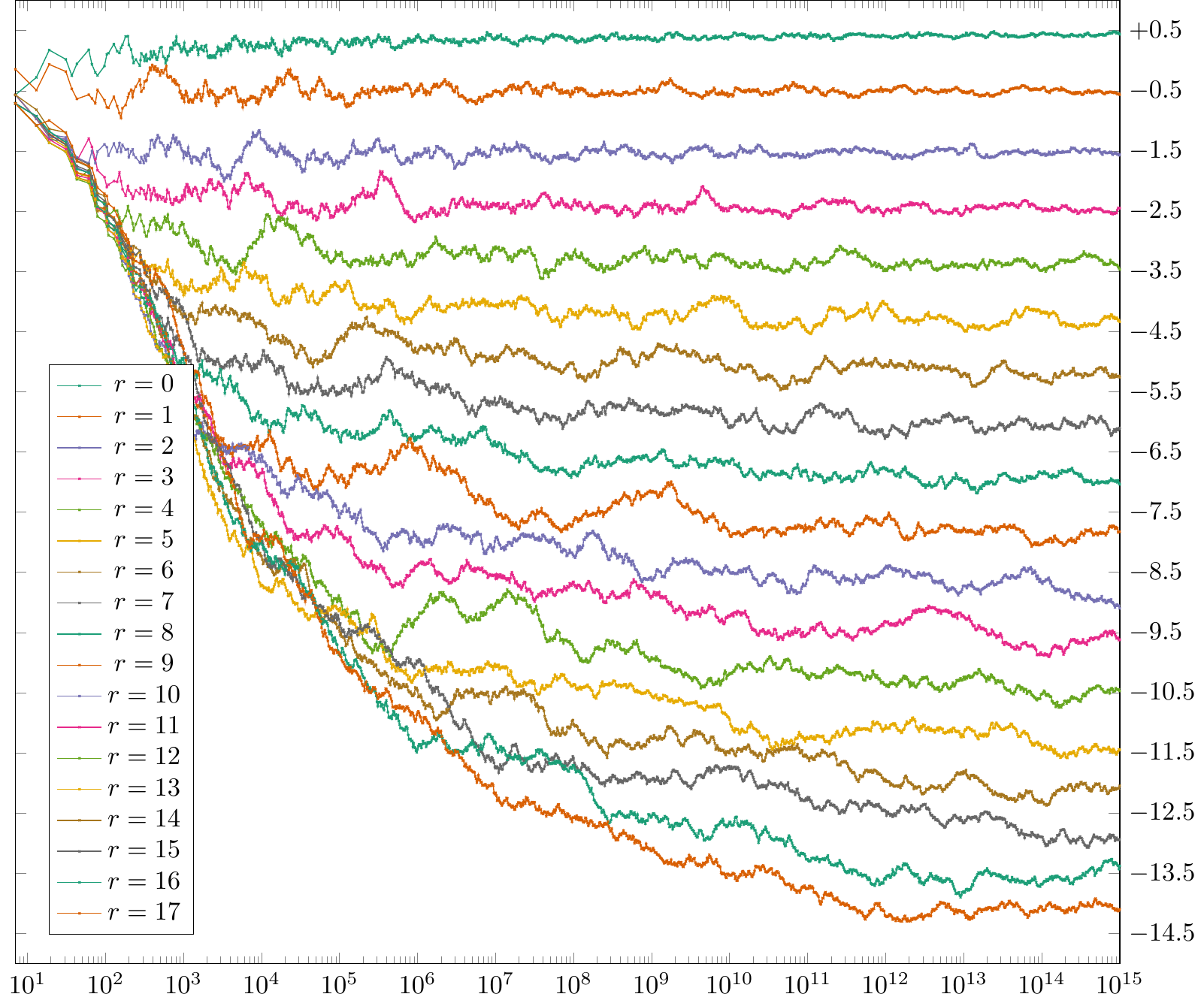}
\vspace{-12pt}
\caption{$S(x)$ plot for Mordell curves $y^2=x^3+k$ of rank $r=r_E$ listed in Table~\ref{tab:sources}.}\label{fig:two}
\vspace{10pt}
\end{figure}

\begin{figure}[b]
\includegraphics[scale=0.95]{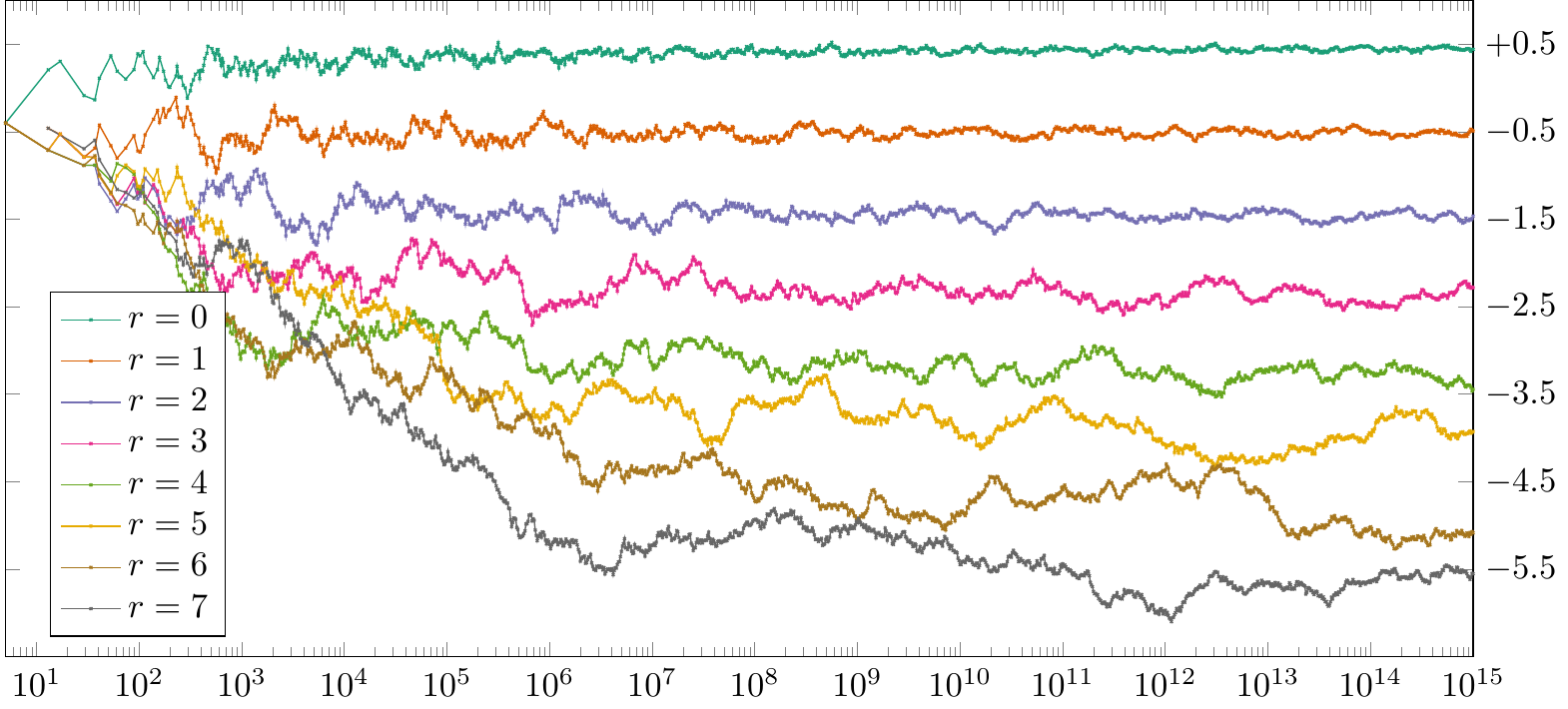}
\vspace{-12pt}
\caption{$S(x)$ plot for congruent number curves $y^2=x^3-n^2x$ of rank $r=r_E$ listed in Table~\ref{tab:sources}.}\label{fig:three}
\end{figure}

\FloatBarrier

\end{document}